\documentclass[11pt, oneside]{amsart} 
\usepackage{graphicx} 
\usepackage[utf8]{inputenc}
\usepackage[english]{babel}
\usepackage[]{amsthm} 
\usepackage[]{amssymb} 
\usepackage{amsmath}
\usepackage{tikz-cd}
\usepackage{bbm}
\usepackage{comment}
\usepackage{thmtools}

\usepackage[margin=0.975in]{geometry}

\usepackage{newpxtext,newpxmath}

\usepackage[
backend=biber,
style=alphabetic,
sorting=ynt
]{biblatex}

\usepackage{geometry}                		
\usepackage[parfill]{parskip}    			

\usepackage{blkarray}
\usepackage{mathtools}
\usepackage{enumerate}
\usepackage[shortlabels]{enumitem}
\usepackage{tikz}

\usepackage{hyperref}
\hypersetup{
	pdfstartview={XYZ null null 1.00}, 
	pdfpagemode=UseNone, 
	colorlinks,
	breaklinks, 
	linkcolor=blue,
	urlcolor=blue, 
	anchorcolor=blue,
	citecolor=blue
}

\usepackage[capitalise]{cleveref}

\numberwithin{equation}{section}

\theoremstyle{plain}

\crefname{prop}{Proposition}{Propositions}
\newtheorem{prop}[equation]{Proposition}

\newtheorem{theorem}[equation]{Theorem}

\crefname{obs}{Observation}{Observations}

\crefname{cor}{Corollary}{Corollaries}
\newtheorem{cor}[equation]{Corollary}

\theoremstyle{definition}

\crefname{defn}{Definition}{Definitions}

\crefname{example}{Example}{Examples}
\newtheorem{example}[equation]{Example}

\theoremstyle{remark}

\crefname{remark}{Remark}{Remarks}

\crefname{claim}{Claim}{Claims}
\newtheorem{claim}[equation]{Claim}

\declaretheoremstyle[
spaceabove=\topsep, 
spacebelow=6pt,
headfont=\normalfont\itshape,
notefont=\normalfont, notebraces={(}{)},
bodyfont=\normalfont,
postheadspace=4pt,
qed=\mbox{\small$\boxtimes$}
]{claimproofstyle}
\declaretheorem[name={Proof of Claim}, style=claimproofstyle, unnumbered]{pf}

\newcommand{\R}{\mathbb{R}}
\newcommand{\N}{\mathbb{N}}
\newcommand{\Z}{\mathbb{Z}}
\newcommand{\set}[1]{\left\{ #1\right\}}

\newcommand{\cocycle}{\mathfrak{w}}
\newcommand{\dfn}[1]{\textbf{#1}}

\title{Oscillating and nonsummable Radon--Nikodym cocycles along the forward geodesic of measure-class-preserving transformations}
\author{Sasha Bell}
\address[Sasha Bell]{
Mathematics and Statistics,
McGill University, Montréal, 
QC, Canada
}
\email{sasha.bell@mail.mcgill.ca}
\author{Tasmin Chu}
\address[Tasmin Chu]{
Mathematics and Statistics,
McGill University, Montréal, 
QC, Canada
}
\email{tasmin.chu@mail.mcgill.ca}
\author{Owen Rodgers}
\address[Owen Rodgers]{
Mathematics and Statistics,
McGill University, Montréal, 
QC, Canada
}
\email{owen.rodgers@mail.mcgill.ca}

\thanks{
S.~Bell was partially supported by an NSERC Undergraduate Student Research Award, with supplemental funding from FRQNT. 
O.~Rodgers and T.~Chu were partially supported by McGill Science Undergraduate Research Awards. 
The authors were partially supported by Anush Tserunyan’s NSERC Discovery Grant RGPIN-2020-07120.
}

\date{\vspace{-5ex}}
\date{} 

\begin{document}

\begin{abstract}
    We consider the least-deletion map on the Cantor space, namely the map that changes the first $1$ in a binary sequence to $0,$ and construct product measures on $2^\N$ so that the corresponding Radon--Nikodym cocycles oscillate or converge to zero nonsummably along the forward geodesic of the map.  
    These examples answer two questions of Tserunyan and Tucker-Drob. 
    We analyze the oscillating example in terms of random walks on $\Z$, using the Chung--Fuchs theorem.
\end{abstract}
\maketitle

\section{Introduction}

This paper focuses on descriptive dynamics of countable Borel equivalence relations (\dfn{CBERs}) on a standard Borel space $X$.
Of particular interest are those equivalence relations which admit an acyclic graphing, called a treeing. 
Such CBERs are called \dfn{treeable.}
In the probability-measure-preserving (pmp) setting, Adams' dichotomy gives a complete characterization of which treeings are amenable: a pmp treeing is amenable if and only if almost every connected component has at most two ends.

This dichotomy fails in general in the measure-class-preserving (mcp) setting, but Tserunyan and Tucker-Drob proved the following generalized version of the Adams dichotomy. Their characterization of amenable treeings is in terms of the Radon--Nikodym cocycle associated to the equivalence relation and the measure: if $E$ is an mcp CBER on a standard probability space \((X,\mu)\), then \(E\) admits a unique Borel (up to null sets) \dfn{Radon--Nikodym cocycle} \(\cocycle : E \to \R^+\) which quantifies the failure of invariance of the measure $\mu$. We think of $\cocycle_y(x):=\cocycle(x,y)$ as the weight of $x$ relative to the weight of $y$.

Tserunyan and Tucker-Drob characterize the amenability of an mcp treeing via the behaviour of the corresponding Radon--Nikodym cocycle. Namely, they show in \cite{TTD} that a locally countable acyclic mcp graph $T$ is amenable if and only if almost every connected component $C$ has at most two \dfn{nonvanishing ends}. Here, an end $\eta \in \partial_T{C}$ is called nonvanishing if for some \(x \in C\), there is \(\varepsilon > 0\) such that for every neighbourhood of the end, there is a point \(y\) in that neighbourhood such that \(\cocycle_x(y) \geq \varepsilon\). 
Their analysis of the behaviour of the Radon--Nikodym cocycle along the ends relies on a close examination of the case when the equivalence relation is given by a function.

More precisely, Tserunyan and Tucker-Drob investigate the Radon--Nikodym cocycle in the setting of a probability space $(X,\mu)$ equipped with a Borel transformation $f: X \rightarrow X$, whose (orbit) equivalence relation $E_f$ is mcp. In particular, they consider the least-deletion map $f: 2^\N \rightarrow 2^\N$ on the Cantor space which changes the first $1$ in a sequence to $0$. They note that equipping $2^\N$ with $\text{Bernoulli}(p)$ power measures yields (depending on whether $\frac{p}{1-p}$ is bigger than 1) examples of the Radon--Nikodym cocycle either converging to $0$ along the forward geodesic summably or converging to infinity.

In this context, Tserunyan and Tucker-Drob ask the following two questions: is there an example of a probability space \((X,\mu)\) equipped with a Borel transformation \(f:X\to X\) with an mcp orbit equivalence relation $E_f$ such that the induced Radon--Nikodym cocycle oscillates along the \(f\)-forward geodesic, meaning it gets arbitrarily large and arbitrarily close to zero? 
Is there an example where the cocycle converges to zero along the forward geodesic, but the sum of the cocycle weights along the geodesic is infinite? 
We answer both questions in the affirmative, completing the picture above. 

The existence of a Radon--Nikodym cocycle which converges to $0$ nonsummably along the forward geodesic in particular reveals a sharp distinction between one-to-one and many-to-one mcp transformations. By a classical theorem of \cite[Prop 1.3.1]{aaronson}, if $f: X \rightarrow X$ is a \textbf{bijective} mcp transformation on a standard probability space $(X, \mu)$, the forward geodesic is summable a.e. if and only if the dynamical system is dissipative. As noted in \cite{TTD}, dissipativity is equivalent to the smoothness of the orbit equivalence relation $E_f$, which in turn is equivalent to the Radon--Nikodym cocycle along the forward geodesic converging to $0$ a.e. Thus, an example like 
\cref{nonsummable} is impossible for invertible transformations.
\subsection*{Acknowledgements}
We are grateful to Anush Tserunyan for her mentorship and helpful comments. We especially thank Louigi Addario-Berry for sketching the probabilistic proof that one of the cocycles oscillates as required. We thank Anush Tserunyan and Robin Tucker-Drob for suggesting the least-deletion map as a candidate to analyze. We also thank Nachi Avraham-Re'em for drawing our attention to Aaronson's result and the relevance of the nonsummable cocycle in classical ergodic theory.

\section{Preliminaries}

Let $(X, \mu)$ be a standard probability space and let $Y$ be a standard Borel space. Given a $\mu$-measurable function $f: X \rightarrow Y$, we denote the \dfn{pushforward} of $\mu$ under $f$ by $f_*\mu$ where $f_*\mu(A) = \mu(f^{-1}(A))$ for each Borel set $A\subseteq Y.$ 
For a fixed element $x\in X$, we call the path $(x, fx, f^2x, \cdots)$ the \textbf{forward geodesic} of $x$. 

We say $g:X \rightarrow X$ is \dfn{measure-class-preserving} (or non-singular) if $g_*\mu$ is equivalent to $\mu$, i.e. they are both absolutely continuous with respect to each other. 

Let $E$ be an equivalence relation. We say $E$ is a \dfn{countable Borel equivalence relation} if each equivalence class $[x]_E$ is countable and if the subset $E \subseteq X \times X$ is a Borel subset of $X \times X$. In this case, notice $E$ itself is a standard Borel space.

For a subset \(A\) of a standard space $X$ endowed with a countable Borel equivalence relation $E$, we denote the \dfn{saturation} of $A$ by
\([A]_{E} = \{x \in X: xEa \text{ for some } a \in A\}.\)

We say a measure $\mu$ defined on the Borel sets of $X$ is \dfn{$E$-quasi-invariant} if \([A]_{E}\) is $\mu$-null whenever \(A\) is $\mu$-null. Equivalently in this case, we say that the measured equivalence relation $E$ on $(X, \mu)$ is \dfn{measure-class-preserving (mcp)}.

By \cite[Section 8]{kechrismiller}, if $E$ is a CBER on a standard probability space \((X,\mu)\) and \(\mu\) is a \dfn{\(E\)-quasi-invariant} measure, then \(E\) admits a unique Borel (up to null sets) \dfn{Radon--Nikodym cocycle} \(\cocycle : E \to \R^+\) satisfies the \dfn{(tilted) mass transport principle}, meaning for any Borel partial bijection \(\gamma : X \rightharpoonup X\) and any \(f\in L^1(X,\mu),\)
\[
\int_{\text{dom}(\gamma)} f(x)d\mu(x) = \int_{\text{im}(\gamma)}f(\gamma(x)) \cocycle_x (\gamma(x)) d\mu(x),
\]
where $\cocycle_x (\gamma(x)) \coloneqq \cocycle(x, \gamma(x))$.

A more comprehensive reference for the dynamics of countable Borel equivalence relations can be found in \cite{kechrismiller}.

\section{The least-deletion map and the corresponding Radon--Nikodym cocycle}

Our setting is that of a countable Borel equivalence relation $E$ on the Cantor space $2^\N$ equipped with an $E$-quasi-invariant probability measure $\mu$. 

Let $X$ be the (co-countable) subset of $2^\N$ of sequences containing infinitely many $1$s.
Consider now the least-deletion map \(f: X \rightarrow X\) which flips the first $1$ in a sequence $x = (x_0, x_1, \cdots) \in 2^\N$ to a $0$. Note that $f$ is finite-to-one. 

We denote by $E_f$ the orbit equivalence relation of $f$. 
Note that $E_f$ coincides with $E_0$ on $X$, the equivalence relation of eventual equality, i.e. $xE_0 y$ iff $x_n = y_n$ for all large enough $n$.
Thus $E_f$ is generated by the involutions $b_n : X \rightarrow X$ which flip the $n$th bit. 
This implies that any product measure $\mu = \prod_{n=1}^ \infty m_n$ on $2^\N$ with non-trivial marginal measures $m_n$ is $E_f$-quasi-invariant since each of these involutions preserves the measure class of $\mu$ by definition. 

For any such $\mu,$ we compute the Radon--Nikodym cocycle \(\cocycle: E_f \rightarrow \R^+\) explicitly, and we derive a formula for $\cocycle_x(f^k(x))$ for $k \in \N$.

\begin{prop}\label{Radon--Nikodym thm}
Fix a non-trivial product measure $\mu = \prod_{n=1}^\infty m_n$ on $2^\N$, and let $\cocycle$ be the Radon--Nikodym cocycle associated to $E_f$ and $\mu$.
Then for each $k \ge 1$ and for a.e. $x \in X$, we have
\[
\cocycle_x(f^k(x)) = \prod_{i=0}^{k-1} \dfrac{m_{n_i}(0)}{m_{n_i}(1)},
\]
where $(n_i)_{i \geq 0}$ is the subsequence of indices $n_i$ for which $x_{n_i}=1$.
\end{prop}
\begin{proof}
    It is enough to show that for a bit-flip involution \(b_n\), we have that 
    \[
    \cocycle_x(b_n (x)) = \frac{1-m_n(x_n)}{m_n(x_n)}.
    \]
    Then, using the cocycle property and that \(f\) always flips a one, we get the desired formula.
    It suffices to show this formula holds on cylinder sets.
    
    Consider a cylinder set \(C_n = [a_0 \cdots a_n]\) with bits \(a_i \in \set{0,1}\) specified up to the nth position. 
    We have \(\mu(C_n) = m_0(a_0) \cdot \ldots \cdot m_n(a_n)\). 
    Then \(b_n(C_n) = [a_0\ldots (1-a_n)]\) so \(\mu(b_n(C_n)) = m_0(a_0) \cdot \ldots \cdot m_{n-1}(a_n)\cdot(1-m_n(a_n))\).
    Computing the Radon--Nikodym derivative of \(b_n\) on this cylinder set, we obtain 
    \[\frac{d(b_{n \ast} \mu)}{d\mu} (C_n) = \frac{m_0(a_0)\cdot \ldots \cdot (1-m_n(a_n))}{m_0 \cdot \cdot \ldots \cdot m_n(a_n)} = \frac{1-m_n(a_n)}{m_n(a_n)},
    \]
    as desired.
\end{proof}

We also record the following fact for which we could not find any reference in the literature.

\begin{prop}
    For any product measure $\mu = \prod_{n=1}^ \infty m_n$ on $2^\N$ with non-trivial marginal measures $m_n,$ the equivalence relation $E_0$ (thus also the least-deletion map $f$) is ergodic with respect to $\mu$.
\end{prop}

\begin{proof}[Proof \emph{(Tserunyan)}]
    Let $A$ be an $E_0$-invariant positively-measured set and let $\varepsilon>0$. By the regularity of $\mu$, there is a cylinder set $[w]$, where $w \in 2^{n}$, such that $\frac{\mu(A \cap [w])}{\mu([w])} > 1 - \varepsilon$. Then it follows from the definition of the product measure and $E_0$-invariance of $A$ that  $\frac{\mu(A \cap [w])}{\mu([w])} = \frac{\mu(A \cap [u])}{\mu([u])}$ for any other word $u \in 2^n$. This implies $\mu(A) > 1 - \varepsilon.$
\end{proof}

\section{Examples of oscillating cocycles}

In \cite{TTD}, Tserunyan and Tucker-Drob ask if it is possible to have a measure-class-preserving function $g: (Y, \nu) \rightarrow (Y, \nu)$ on a standard probability space $(Y, \nu)$ such that the Radon--Nikodym cocycle with respect to $\mu$, defined on the orbit equivalence relation induced by $g$, \textbf{oscillates} along the \(g\)-forward geodesic: that is \(\displaystyle \limsup_{k\rightarrow \infty}\cocycle_y(g^{k}(y)) = \infty\)
and 
\(\displaystyle\liminf_{k\rightarrow \infty}\cocycle_y (g^{k}(y)) = 0
\) for a.e.\ $y \in Y$.

Letting $f : X \rightarrow X$ denote the least-deletion map defined above, we provide examples of product measures on $2^\N$ so that their Radon--Nikodym cocycles exhibit the desired oscillatory behaviour. Recall from above that $X$ is the subset of $2^\N$ of sequences with infinitely many 1s.

\begin{example}[Oscillating cocycle]\label{oscillating period 3}
Consider the product measure $\mu:=\prod_{n=0}^\infty m_n$ defined as follows.
If $n \equiv 0$ mod $3$, set 
\[
m_n(0):= 1/3\text{ and } m_n(1):=2/3,
\]
so $m_n$ is biased towards $1$. If $n \not \equiv 0$ mod $3$, set
\[
m_n(0):= 2/3\text{ and }m_n(1):=1/3,
\]
so $m_n$ is biased towards $0$.
\end{example}

\begin{theorem}\label{oscillating theorem}
With the measure $\mu$ as above, the Radon--Nikodym cocycle $\cocycle$ oscillates along the forward geodesic of the least-deletion map $f$ for a.e.\ $x \in X.$
\end{theorem}

\begin{proof}

Fix $x \in X$ and for each $k\geq 1,$ let $C_k(x) : = \cocycle_x(f^k(x))$. By our computation of the Radon--Nikodym cocycle,
\[
C_k(x) = \prod_{i=1}^k \frac{m_{n_i}(0)}{m_{n_i}(1)}
\]
where $(n_i)_{i \geq 1}$ is the subsequence of $\N$ for which $x_{n_i}=1$. 
Call $C(x):=(C_k(x))_{k \geq 0}$ the cocycle sequence at $x.$ 

We now define a "lazy" version of $C(x),$ denoted $\Tilde{C}(x)$. We set $\Tilde{C}_0(x):=1$, and define $\Tilde{C}(x)$ inductively by
\[
\Tilde{C}_{k+1}(x) := \begin{cases}
    \frac{m_k(0)}{m_k(1)}\Tilde{C}_k(x) &  \text{if 
    } x_k =1 \\
    \Tilde{C}_k(x) & \text{if } x_k = 0
\end{cases}
\]

Notice that $\Tilde{C}_k$ jumps precisely when $k \in (n_i)_{i \geq 1}$, and otherwise it lingers at the same value. We will prove that $\tilde{C}(x)$ oscillates and thus conclude that the cocycle sequence $C(x)$ oscillates.

For each $k \in \N$, observe that $\tilde{C}_k(x)$ is a power of 2. We will write $\tilde{C}_k$ instead of $\tilde{C}_k(x)$, since $x$ is fixed.
Let $L_k := \log_2 \tilde{C}_k$. Then $L_{k} \in \mathbb{Z}$ for all $k$. Now set $X_0 := 0$ and let 
\[
Z_k := L_{3k} - L_{3(k-1)}
\]
for $k > 0$. Notice that $Z_k$ for $k \geq 1$ takes values in $\{-1,0,1,2\}$ deterministically according to the value of the sequence $x$ on the indices $3k-2, 3k-1,$ and $3k$.

We now treat $x$ and $Z_k(x)$ as random variables, where $x$ is distributed according to $\mu$. 
Notice that
\[
Z_k(x) =\mathbbm{1}_{\{x_{3k-2}=1\}} + \mathbbm{1}_{\{x_{3k-1}=1\}} - \mathbbm{1}_{\{x_{3k}=1\}}.
\]
Since the product measure was periodically defined for triples, the $Z_k$ are independent identically distributed random variables.

By the linearity of expectation,
\begin{align*}
    \mathbb{E}[Z_k] &= \mathbb{P}(x_{3k-2}=1) + \mathbb{P}(x_{3k-1}=1) - \mathbb{P}(x_{3k}=1)
    = \frac{1}{3} + \frac{1}{3} - \frac{2}{3} = 0.
\end{align*}
Since $Z_k = L_{3k} - L_{3(k-1)}$, we can write 
\(
L_{3k} = \sum_{i=1}^k Z_i,
\)
so we have that $L_{3k}$ is a random walk on \(\Z\) with zero bias. A theorem by Chung and Fuchs \cite{chung} says that under mild conditions which are satisfied above, random walks on $\Z$ with zero bias oscillate almost surely. 

Thus, for almost every sequence $x\in X,$
\[
\limsup_{k\to\infty} L_{3k}(x) = \infty \quad \text{and} \quad \liminf_{k\to\infty} L_{3k}(x) = -\infty,
\]
so exponentiating we find that $\mu$-a.e.,
\[
\limsup_{k\to\infty} \Tilde{C}_k = \infty \quad 
\text{and} \quad 
\liminf_{k\to\infty} \Tilde{C}_k = 0.
\]
This gives us that the lazy cocycle sequence $\Tilde{C}(x)$ oscillates almost surely. Thus $C(x)$ must also oscillate.
\end{proof}

\example[An infinite family of oscillating cocycles]\label{Oscillation with period m}\label{generaloscillating}
We can generalize the above example to obtain a family of periodic measures, all of which have oscillating cocycles.

Fix $j>2$ and consider the infinite product measure $\mu:=\prod_{n=0}^\infty m_n$ defined as follows.
If $n \equiv 0$ mod $j$, set 
\[
m_n(0):= \frac{1}{j}\text{ and }m_n(1):=\frac{j-1}{j},
\]
and if $n \not \equiv 0$ mod $j$, set
\[
m_n(0):= \frac{j-1}{j}\text{ and }m_n(1):=\frac{1}{j}.
\]
Let $\cocycle$ be the cocycle of $\mu$ with respect to the least-deletion map $f:X\rightarrow X.$ The proof of example \labelcref{oscillating theorem} can be straightforwardly generalized to give the following corollary.  
\begin{cor}
With the measure $\mu$ as above, the Radon--Nikodym cocycle $\cocycle$ oscillates along the forward geodesic of the least-deletion map $f$ for a.e.\ $x \in X.$
\end{cor}

\section{Example of cocycle vanishing nonsummably along the forward geodesic}

In \cite{TTD}, Tserunyan and Tucker-Drob ask if it is possible to have a measure-class-preserving function $g: (Y, \nu) \rightarrow (Y, \nu)$ on a standard probability space $(Y, \nu)$ such that the Radon--Nikodym cocycle with respect to $\nu$, defined on the orbit equivalence relation induced by $g$, \textbf{vanishes nonsummably} along the forward geodesic. That is,
\[
\lim_{k\rightarrow\infty}\cocycle_y(g^k(y)) = 0 \quad \text{and} \quad
\sum_{k\in\N} \cocycle_y(g^k(y)) = \infty \quad
\text{ for $\nu$-a.e. } y\in Y.\]

We provide an example of a cocycle approaching zero nonsummably, on the space $X \subseteq 2^\N$ of sequences with infinitely many 1s, with the least-deletion map $f: X \rightarrow X$.

\example{(Cocycle going to zero nonsummably).}\label{nonsummable}
Let $N:=(n_k)_{k\geq 1}$ be a (sparse) infinite subsequence of $\N$ defined as follows: set $n_0:=0,$ and for $k \geq 1$, set
\[
\quad n_{k+1} := n_k + 2^{k\sum_{i\leq k} 
i}.
\]
We construct a product measure $\mu = \prod_{n=1}^\infty m_n$ as follows. For $n_k\in N,$ set
\[
m_{n_k} (0) := \frac{1}{2^k+1} \text{ and } m_{n_k}(1) = 
\frac{2^k}{2^k+1},
\]
and otherwise set
\[
m_n(0)\coloneqq m_n(1)\coloneqq \frac{1}{2}.
\]
Then 
\[
\frac{m_{n}(0)}{m_{n}(1)} = 
\begin{cases}
\frac{1}{2^k} & \text{ if } n = n_k \in N
\\
1 & \text{ otherwise.}
\end{cases}
\]
Let $\mathfrak{w}$ be the cocycle of $\mu$ with respect to the least-deletion map $f:X\rightarrow X.$
\begin{theorem}
    The cocycle $\cocycle$ vanishes nonsummably along the forward geodesic for almost every $x \in X.$
\end{theorem}
\begin{proof}
Notice that the set
\(
\{x\in X:x_{n_k} = 0\text{ for infinitely many }k\}
\)
is null by Borel--Cantelli, so the set on which the cocycle vanishes along the forward geodesic is conull.

Let
\(
p_k:=2^{\sum_{i\leq k} i}
\)
for all $k\geq 1.$ Then $n_{k+1} - n_k = p_k^k$ for all $k\geq 1.$ 
Let
\(
Z_k:=\{x\in X:\text{the number of 1s in $x$ between indices }n_k\text{ and }n_{k+1}\text{ is }\leq p_k\},
\) 
and let 
\(
Z \displaystyle\coloneqq \{x\in X:x\text{ is in infinitely many }Z_k\}.
\)

Now we require the following:

\begin{claim}\label{E claim}
\(\displaystyle
\{x\in X:x\text{ has a summable cocycle along the forward geodesic}\} \subseteq Z.
\)
\end{claim}
\begin{pf}
    For a given $x \in X$ and $k \in \N,$ suppose that the number of $1$s between $n_k$ and $n_{k+1}$ is $>p_k.$ Let $i:=\sup\{j<n_{k+1}:x_j=1\}.$ Then 

\[
\cocycle_x(F^i(x)) = \prod_
{\substack{j<k\\ x_{n_j}=1}}\frac{1}{2^j} \geq \prod_{j<k}\frac{1}{2^j} = 2^{-\sum_{j < k}j},
\]
so 
\[
\sum_{\substack{n_k<i<n_{k+1}\\x_i =1}}
\cocycle_x(F^i(x)) \geq p_k 2^{-\sum_{j < k}j} = 1.
\]
Hence if $x$ is summable, it must be in all but finitely many $Z_k$, since every membership in a $Z_k^c$ contributes at least $1$ to the sum.
\end{pf}

We show that $Z$ is null, and  conclude that the cocycle is summable only on a null set. 

For a fixed $k$,
\begin{align*}
\mu(Z_k) &= \frac{1}{2^{p_k^k}}\sum_{i=0}^{p_k}{{p_k^k} \choose i}\leq \frac{1}{2^{{p_k^k}}} \sum_{i=1}^{p_k} \frac{p_k^{ik}}{i!} 
\leq \frac{1}{2^{p_k^k}} p_k \frac{p_k^{kp_k}}{p_k!} \leq \frac{1}{2^{p_k^k}}p_k^{kp_k}\\
& = \frac{1}{2^{p_k^k}}2^{kp_k\sum_{i\leq k}i } = 2^{kp_k\sum_{i\leq k}i - p_k^k}\leq 2^{-k}
\end{align*}
for all but finitely many $k.$
Hence,
\(
\mu(Z_k) \leq \dfrac{1}{2^k}
\)
for all sufficiently large $k$, so $\sum_{k\in\N}\mu(Z_k) < \infty.$ 

The events $(Z_k)_{k\geq 1}$ are independent since they concern disjoint sets of indices, so \(\mu(Z) = 0\) by Borel--Cantelli. 
\end{proof}

\printbibliography

\end{document}